\documentclass[12pt]{amsart}

\usepackage{amsmath}
\usepackage{amssymb}
\usepackage{array}

\newtheorem{theorem}{Theorem}
\newtheorem{lemma}{Lemma}
\newtheorem{corollary}{Corollary}
\newtheorem{conjecture}{Conjecture}
\newtheorem{proposition}{Proposition}

\theoremstyle{definition}

\newtheorem{remark}{{\rm\bf Remark}}[theorem]

\begin{document}

\title{Congruences arising from Ap\'ery-type series for zeta values}

\author{Kh.~Hessami Pilehrood}
\address{\begin{flushleft}
Department of Mathematics and Statistics \\ Dalhousie University \\ Halifax, Nova Scotia, B3H 3J5, Canada
\end{flushleft}}






\author{T.~Hessami Pilehrood}

\email{hessamik@gmail.com,  hessamit@gmail.com}


\date{}


\begin{abstract}
Recently, R.~Tauraso established finite $p$-analogues of  Ap\'ery's famous series for
$\zeta(2)$ and $\zeta(3).$
In this paper, we present several congruences for finite central binomial sums
arising from the truncation of Ap\'ery-type series for $\zeta(4)$ and $\zeta(5).$
We also prove a $p$-analogue of Zeilberger's series for $\zeta(2)$ confirming a conjecture
of Z.~W.~Sun.

\end{abstract}

\maketitle

\section{Introduction}
\label{intro}

The Riemann zeta function for ${\rm Re}\,s>1$ is defined by the series
$\zeta(s)=\sum_{n=1}^{\infty}1/n^s.$
It is well-known,  due to Euler, that the value of the
Riemann zeta function at an even positive integer can be expressed in terms of $\pi.$
Namely,
$$
\zeta(2m)=(-1)^{m-1} (2\pi)^{2m} \frac{B_{2m}}{2 (2m)!}, \qquad\qquad m\in {\mathbb N},
$$
where $B_{m}\in {\mathbb Q}$ are   Bernoulli numbers defined by the series expansion
$$
\frac{t}{e^t-1}=\sum_{m=0}^{\infty}B_m \frac{t^m}{m!},
$$
which yields  $B_0=1,$ $B_1=-1/2,$ $B_2=1/6,$ and $B_{2m+1}=0,$ $m\in{\mathbb N}.$

The alternating multiple harmonic sum  is defined as
$$
H_n(a_1, a_2, \dots, a_r)=\sum_{1\le k_1<k_2<\dots<k_r\le n}\prod_{i=1}^r
\frac{{\rm sgn}(a_i)^{k_i}}{k_i^{|a_i|}},
$$
where $n\ge 0,$ $r\ge 1$ and $(a_1, a_2, \ldots, a_r)\in({\mathbb Z}^{*})^r$ (here and in the sequel an empty
sum is treated as zero).

Recently, R.~Tauraso \cite{t1} showed that  Ap\'ery's famous series for $\zeta(3)$ and $\zeta(2),$
\begin{equation}
\zeta(3)=\frac{5}{2}\sum_{k=1}^{\infty}\frac{(-1)^{k-1}}{k^3\binom{2k}{k}}\qquad\text{and}\qquad
\zeta(2)=3\sum_{k=1}^{\infty}\frac{1}{k^2\binom{2k}{k}},
\label{eq02}
\end{equation}
that were used in his irrationality proofs \cite{po} of these numbers admit very nice $p$-analogues:
\begin{equation}
\frac{5}{2}\sum_{k=1}^{p-1}\frac{(-1)^{k-1}}{k^3\binom{2k}{k}}\equiv \frac{H_{p-1}(1)}{p^2} \quad
\text{and} \quad
3\sum_{k=1}^{p-1}\frac{1}{k^2\binom{2k}{k}}\equiv \frac{H_{p-1}(1)}{p}\pmod{p^3},
\label{eq02.5}
\end{equation}
where $p>5$ is a prime.
The similar Ap\'ery-type series for $\zeta(4)$ is also well-known (see \cite[p.\ 89]{com}):
\begin{equation}
\zeta(4)=\frac{36}{17}\sum_{k=1}^{\infty}\frac{1}{k^4\binom{2k}{k}}.
\label{eq03}
\end{equation}
A generalization of formulas (\ref{eq02}), (\ref{eq03}) to odd zeta values $\zeta(2n+3),$ $n\in{\mathbb N},$ with
the help of the generating function identity
\begin{equation}
\sum_{n=0}^{\infty}\zeta(2n+3)a^{2n}=\sum_{k=1}^{\infty}\frac{1}{k(k^2-a^2)}=
\frac{1}{2}\sum_{k=1}^{\infty}\frac{(-1)^{k-1}}{k^3\binom{2k}{k}}\,\,\frac{5k^2-a^2}%
{k^2-a^2}\,\prod_{m=1}^{k-1}\left(1-\frac{a^2}{m^2}\right),
\label{eq04}
\end{equation}
where $a\in{\mathbb C},$ $|a|<1,$ was given by Koecher \cite{ko} (and independently in an expanded form by
Leshchiner \cite{le}).
Expanding the right-hand side of (\ref{eq04}) in powers of $a^2$ and
comparing coefficients of $a^{2n}$ on both sides of (\ref{eq04}) gives the
Ap\'ery-like series for $\zeta(2n+3).$
In particular, comparing constant terms   recovers  formula (\ref{eq02}) for $\zeta(3)$
and comparing coefficients of $a^2$  gives the
following  formula for $\zeta(5):$
\begin{equation}
\zeta(5)=2\sum_{k=1}^{\infty}\frac{(-1)^{k-1}}{k^5\binom{2k}{k}}-\frac{5}{2}\sum_{k=1}^{\infty}
\frac{(-1)^{k-1}H_{k-1}(2)}{k^3\binom{2k}{k}}.
\label{zeta5}
\end{equation}
First results related to generating function identities for even zeta values
belong to Leshchiner \cite{le} who proved (in an expanded form)
that for $|a|<1,$
\begin{equation}
\sum_{n=0}^{\infty}\left(1-\frac{1}{2^{2n+1}}\right)\zeta(2n+2)a^{2n}=\sum_{n=1}^{\infty}
\frac{(-1)^{n-1}}{n^2-a^2}=\frac{1}{2}\sum_{k=1}^{\infty}\frac{1}{k^2\binom{2k}{k}}\,\,
\frac{3k^2+a^2}{k^2-a^2}\,\prod_{m=1}^{k-1}\left(1-\frac{a^2}{m^2}\right).
\label{eq06}
\end{equation}
Comparing  constant terms on both sides of (\ref{eq06}) implies formula (\ref{eq02}) for
$\zeta(2)$ and comparing coefficients of $a^2$ yields
\begin{equation}
\zeta(4)=\frac{16}{7}\sum_{k=1}^{\infty}\frac{1}{k^4\binom{2k}{k}}
-\frac{12}{7}\sum_{k=1}^{\infty}\frac{H_{k-1}(2)}{k^2\binom{2k}{k}}.
\label{eq07}
\end{equation}
From (\ref{eq03}), (\ref{eq07}) we get easily  the following reduction formula:
$$
\sum_{k=1}^{\infty}\frac{H_{k-1}(2)}{k^2\binom{2k}{k}}=\frac{5}{51}\sum_{k=1}^{\infty}
\frac{1}{k^4\binom{2k}{k}}=\frac{5}{108}\zeta(4).
$$
In 2006,
 D.~Bailey, J.~Borwein and D.~Bradley  \cite{bbb} proved another identity
\begin{equation}
\sum_{n=0}^{\infty}\zeta(2n+2)a^{2n}=
\sum_{k=1}^{\infty}\frac{1}{k^2-a^2}=
3\sum_{k=1}^{\infty}\frac{1}{\binom{2k}{k}(k^2-a^2)}
\prod_{m=1}^{k-1}\left(\frac{m^2-4a^2}{m^2-a^2}\right).
\label{2n2}
\end{equation}
It generates similar Ap\'ery-like series for the numbers $\zeta(2n+2),$
which are not covered by Leshchiner's result (\ref{eq06}). In particular, for
$\zeta(4)$ it gives
\begin{equation}
\zeta(4)=3\sum_{k=1}^{\infty}\frac{1}{k^4\binom{2k}{k}}-
9\sum_{k=1}^{\infty}\frac{H_{k-1}(2)}{k^2\binom{2k}{k}}.
\label{zz4}
\end{equation}
In this paper, we prove $p$-analogues of  Ap\'ery-type series
for $\zeta(4)$ and  $\zeta(5)$  arising from the truncation of the series
(\ref{eq03}), (\ref{zeta5}), (\ref{eq07}) and (\ref{zz4}).
\begin{theorem} \label{t1}
Let $p>3$ be a prime. Then we have
$$
4\sum_{k=1}^{p-1}\frac{1}{k^4\binom{2k}{k}}-3\sum_{k=1}^{p-1}\frac{H_{k-1}(2)}{k^2\binom{2k}{k}}
\equiv \frac{3}{p^3}H_{p-1}(1)-\frac{6}{5}pB_{p-5} \pmod{p^2}.
$$
\end{theorem}
\begin{theorem} \label{t2}
Let $p>3$ be a prime. Then
\begin{equation*}
\sum_{k=1}^{p-1}\frac{1}{k^4\binom{2k}{k}}  \equiv
\frac{H_{p-1}(1)}{p^3} \pmod{p}
\end{equation*}
and
\begin{equation*}
\sum_{k=1}^{p-1}\frac{\binom{2k}{k}}{k^3}  \equiv
\frac{-2H_{p-1}(1)}{p^2} \pmod{p}.
\end{equation*}
\end{theorem}
In \cite[Conj.~1.1]{sun1} Z.~W.~Sun conjectured that for each prime $p>5,$
\begin{equation}
\sum_{k=1}^{p-1}\frac{1}{k^4\binom{2k}{k}}  \equiv
\frac{H_{p-1}(1)}{p^3}-\frac{7}{45}pB_{p-5} \pmod{p^2}
\label{eq07.1}
\end{equation}
and for $p>7,$
\begin{equation}
\sum_{k=1}^{p-1}\frac{\binom{2k}{k}}{k^3}  \equiv
\frac{-2H_{p-1}(1)}{p^2}-\frac{13}{27}H_{p-1}(3) \pmod{p^4}.
\label{eq07.2}
\end{equation}
Theorem \ref{t2} confirms (\ref{eq07.1}), (\ref{eq07.2}) modulo a prime. Moreover,
from Theorem \ref{t1} it follows that (\ref{eq07.1}) is equivalent to the following
\begin{conjecture}
Let $p>3$ be a prime. Then
\begin{equation}
\sum_{k=1}^{p-1}\frac{H_{k-1}(2)}{k^2\binom{2k}{k}}\equiv \frac{H_{p-1}(1)}{3p^3}+\frac{26}{135}pB_{p-5}
\pmod{p^2}.
\label{eq07.3}
\end{equation}
\end{conjecture}
In this paper we prove (\ref{eq07.3}) modulo a prime (see Lemma \ref{l1} below).
\begin{corollary}
Let $p>5$ be a prime. Then we have
$$
\sum_{k=1}^{p-1}\frac{1}{k^4\binom{2k}{k}}-3\sum_{k=1}^{p-1}\frac{H_{k-1}(2)}{k^2\binom{2k}{k}}\equiv 0\pmod{p}.
$$
\end{corollary}
\begin{theorem} \label{t3}
Let $p>3$ be a prime. Then we have
$$
2\sum_{k=1}^{p-1}\frac{(-1)^{k-1}}{k^5\binom{2k}{k}}-\frac{5}{2}\sum_{k=1}^{p-1}
\frac{(-1)^{k-1} H_{k-1}(2)}{k^3\binom{2k}{k}}\equiv -\frac{6}{5}B_{p-5} \pmod{p}.
$$
\end{theorem}
\begin{theorem} \label{t4}
Let $p>5$ be a prime. Then
$$
p\sum_{k=1}^{p-1}\frac{(-1)^{k-1}}{k^5\binom{2k}{k}}\equiv \frac{5(1-L_p^2)}{8p^4}-\frac{5}{4p^3}
({\mathcal L}_1(\varphi^2)+{\mathcal L}_1(\varphi^{-2}))+\frac{3H_{p-1}(1)}{p^3}-\frac{3p}{5}B_{p-5} \pmod{p^2},
$$
and
$$
\sum_{k=1}^{p-1}\frac{(-1)^{k}\binom{2k}{k}}{k^4}\equiv \frac{5(1-L_p^2)}{4p^4}-\frac{5}{2p^3}
({\mathcal L}_1(\varphi^2)+{\mathcal L}_1(\varphi^{-2}))+\frac{6H_{p-1}(1)}{p^3} \pmod{p},
$$
where $L_n$ is the $n$th Lucas number defined by the recurrence $L_0=2,$ $L_1=1,$ $L_n=L_{n-1}+L_{n-2},$ $n>1,$
$\varphi=\frac{1+\sqrt{5}}{2}$ is the golden ratio, and
${\mathcal L}_1(x)=\sum_{k=1}^{p-1}x^k/k$ is the finite $1$-logarithm.
\end{theorem}
The first congruence of Theorem \ref{t4} gives a finite $p$-analogue of the following
alternating central binomial sum  that was evaluated explicitly in \cite{li}:
$$
\sum_{k=1}^{\infty}\!\frac{(-1)^{k-1}}{k^5\binom{2k}{k}}=\!\!-2\zeta(5)+\frac{5}{2}{\rm Li}_5(\varphi^{-2})
+5 {\rm Li}_4(\varphi^{-2})\log\varphi+4\zeta(3)\log^2\varphi-\frac{8}{3}\zeta(2)\log^3\varphi
+\frac{4}{3}\log^5\varphi,
$$
where ${\rm Li}_n(z)=\sum_{k=1}^{\infty}z^k/k^n$ is the classical polylogarithm.

In \cite{mt}, S.~Mattarei and R.~Tauraso extended results
of \cite{t1} and described a general approach on obtaining  congruences for the finite sums
\begin{equation}
\sum_{k=1}^{p-1}\frac{t^k}{k^d\binom{2k}{k}} \pmod{p^2}, \qquad d=0,1,2,3\footnote{when $d=3$ the congruence
was established only modulo $p$ (see \cite[Cor.~6.4]{mt})},
\label{eq08}
\end{equation}
and
\begin{equation}
p\sum_{k=1}^{p-1}\frac{H_{k-1}(2) t^k}{k^d\binom{2k}{k}} \pmod{p}, \qquad d=0,1,2.
\label{eq09}
\end{equation}
Note that some special cases of sum (\ref{eq09}) were considered earlier by Z.~W.~Sun in \cite{supi}.
The crucial idea of  work \cite{mt} relies on a connection of values (\ref{eq08}), (\ref{eq09}) with
sums of the form
$$
\sum_{k=1}^{p-1}\frac{u_k(2-t)}{k^d} \qquad\text{and}\qquad \sum_{k=1}^{p-1}
\frac{v_k(2-t)}{k^d}
$$
that can be written in terms of the finite polylogarithms (see \cite[\S
8]{mt})
$$
\mathcal{L}_d(x)=\sum_{k=1}^{p-1}\frac{x^k}{k^d}, \qquad d\in {\mathbb N}.
$$
Here $\{u_n(x)\}_{n\ge 0}$ and $\{v_n(x)\}_{n\ge 0}$ are Lucas sequences defined by the
recurrence relations
$$
u_0(x)=0, \quad u_1(x)=1, \quad\text{and}\quad u_n(x)=xu_{n-1}(x)-u_{n-2}(x) \quad\text{for}\quad
n>1,
$$
$$
v_0(x)=2, \quad v_1(x)=x, \quad\text{and}\quad v_n(x)=xv_{n-1}(x)-v_{n-2}(x) \quad\,\text{for}\quad
n>1.
$$
In principle, their method can be generalized to get congruences for sums (\ref{eq08}) with $d\ge 4.$
So, for example, R.~Tauraso communicated us the following identity:
\begin{align*}
\binom{2n}{n}\sum_{k=1}^n\frac{t^k}{k^4\binom{2k}{k}}
&=\sum_{k=1}^n \binom{2n}{n-k}\frac{v_k(t-2)}{k^4}+\binom{2n}{n}
\sum_{k=1}^n\frac{1}{k^4}\\
&+2\sum_{1\leq j<k\leq n} \binom{2n}{n-k}
\left(\frac{1}{k}+\frac{1}{j}\right)\frac{(-1)^{k-j}v_j(t-2)}{jk^2}\\
&+4\sum_{1\leq i< j<k\leq n} \binom{2n}{n-k}
\frac{(-1)^{k-i}v_i(t-2)}{ijk^2}
\end{align*}
that can be obtained by integration from \cite[Th.~5.3]{mt}.
In the easiest case, when $t=4$, then $v_n(t-2)=2$ for $n\geq 0$ and it
follows that (see  \cite[Section 8]{mt})
$$
p\sum_{k=1}^{p-1}\frac{4^k}{k^4\binom{2k}{k}}\equiv -\frac{4}{3}
\left(2q_p(2)^3+B_{p-3}\right)\pmod{p},
$$
where $q_p(2)=\frac{2^{p-1}-1}{p}$ is the Fermat quotient.
For other values of $t,$ in particular, for $t=1$ it is not so easy to derive
similar congruences, since one needs evaluations for values of finite multiple
polylogarithms modulo a power of a prime.
To prove our Theorems \ref{t1}--\ref{t4}, we employ another method based on application
of appropriate WZ pairs that were found in \cite{he1} for demonstrating
identities (\ref{eq04}),  (\ref{eq06}), (\ref{2n2}). The similar approach also allows us to
establish
the following $p$-analogue of Zeilberger's series for
$\zeta(2),$
$$
\zeta(2)=\sum_{k=1}^{\infty}\frac{21k-8}{k^3\binom{2k}{k}^3}.
$$
\begin{theorem} \label{t5}
Let $p$ be a prime greater than $5.$ Then we have
$$
\sum_{k=1}^{p-1}\frac{21k-8}{k^3\binom{2k}{k}^3}+\frac{p-1}{p^3}\equiv
\frac{H_{p-1}(1)}{p^2}(15p-6)+\frac{12}{5}p^2B_{p-5} \pmod{p^3}.
$$
\end{theorem}
Note that Theorem \ref{t5} confirms Z.~W.~Sun's conjecture \cite[(1.19)]{sun1}.

The paper is organized as follows. In Sections \ref{s1.5} and \ref{s2}, we recall some important divisibility
properties of multiple harmonic sums and
 prove some helpful lemmas.
In Section \ref{s3}, we give proofs of Theorems \ref{t1}--\ref{t4}.
In final Section \ref{s4},  we  demonstrate some interesting combinatorial identities and interrelated congruences
that are essential for the proof of Theorem \ref{t5}.

\section{Multiple harmonic sums} \label{s1.5}

We start by recalling some important properties of multiple harmonic sums.
Let $a,b,c$ be positive integers. We will need the following two formulas for the product:
$$
H_n(a)\cdot H_n(b)=H_n(a,b)+H_n(b,a)+H_n(a+b),
$$
$$
H_n(a,b)\cdot H_n(c)=H_n(a,b,c)+H_n(a,c,b)+H_n(c,a,b)+H_n(a+c,b)+H_n(a,b+c).
$$
The divisibility properties of multiple harmonic sums were studied in \cite{hof,zhs}, \cite{t1}--\cite{zhou}
and many of them are related to the Bernoulli numbers:

\vspace{0.3cm}

(a) (\cite[Th.~5.1, Cor.~5.1]{zhs}) for   $a>0$ and for any prime $p\ge a+3,$
\begin{equation*}
H_{p-1}(a)\equiv\begin{cases}
-\frac{a(a+1)}{2(a+2)}\,p^2 B_{p-a-2} \pmod{p^3}     & \quad \text{if} \quad a \quad \text{is odd}, \\[3pt]
\frac{a}{a+1}\,pB_{p-a-1} \qquad\, \pmod{p^2} & \quad\text{if} \quad a \quad \text{is even};
\end{cases}
\end{equation*}

\vspace{0.3cm}

(b) (\cite[Rem.~5.1]{zhs}, \cite[Th.~2.1]{t1}) for any prime $p>5,$
\begin{equation*}
\begin{array}{llll}
H_{p-1}(1)&\equiv &p^2\left(\displaystyle 2\frac{B_{p-3}}{p-3}-\frac{B_{2p-4}}{2p-4}\right) &\pmod{p^4}, \\[12pt]
H_{p-1}(2)&\equiv &\displaystyle-\frac{2H_{p-1}(1)}{p}+\frac{2}{5}p^3B_{p-5} &\pmod{p^4}, \\[10pt]
H_{p-1}(1)&\equiv & p^2\left(\displaystyle\frac{B_{3p-5}}{3p-5}-3\frac{B_{2p-4}}{2p-4}+3\frac{B_{p-3}}{p-3}\right)+p^4
\displaystyle\frac{B_{p-5}}{p-5} &\pmod{p^5}.
\end{array}
\end{equation*}

\vspace{0.3cm}

(c) (\cite{zhou},\cite[Th.~1.6]{zhao}) for $a, r>0$ and for any prime $p>ar+2,$
\begin{equation*}
H_{p-1}(\{a\}^r)\equiv \begin{cases}
(-1)^r\frac{a(ar+1)p^2}{2(ar+2)}B_{p-ar-2} \pmod{p^3}     & \quad \text{if} \quad ar \quad \text{is odd}, \\[3pt]
(-1)^{r-1}\frac{ap}{ar+1}\,B_{p-ar-1} \,\, \pmod{p^2} & \quad\text{if} \quad ar \quad \text{is even};
\end{cases}
\end{equation*}

\vspace{0.3cm}

(d) (\cite[Th.~3.1, 3.2]{zhao}) for $a_1, a_2>0$ and for any prime $p\ge a_1+a_2,$
$$
H_{p-1}(a_1,a_2)\equiv\frac{(-1)^{a_2}}{a_1+a_2}\binom{a_1+a_2}{a_1}B_{p-a_1-a_2} \pmod{p},
$$
moreover, if $a_1+a_2$ is even, then for any prime $p>a_1+a_2+1,$
\begin{equation*}
\begin{split}
H_{p-1}(a_1,a_2)&\equiv p\left[(-1)^{a_1}a_2\binom{a_1+a_2+1}{a_1}-(-1)^{a_1}a_1\binom{a_1+a_2+1}{a_2}-a_1-a_2\right] \\[3pt]
&\times
\frac{B_{p-a_1-a_2-1}}{2(a_1+a_2+1)} \pmod{p^2};
\end{split}
\end{equation*}

\vspace{0.3cm}

(e) (\cite[Th.~3.5]{zhao} if $a_1, a_2, a_3>0$ and $w:=a_1+a_2+a_3$ is odd, then for any prime $p>a_1+a_2+a_3,$
$$
H_{p-1}(a_1,a_2,a_3)\equiv \left[(-1)^{a_1}\binom{w}{a_1}-(-1)^{a_3}\binom{w}{a_3}\right]\frac{B_{p-w}}{2w} \pmod{p};
$$

\vspace{0.3cm}

(f) (\cite[Prop.~3.8]{zhao}, \cite[Th.~5.2, 7.2]{hof}) for any prime $p>5,$
$$
H_{p-1}(1,1,2)\equiv \frac{11}{10}pB_{p-5}, \quad
H_{p-1}(1,2,1)\equiv -\frac{9}{10}pB_{p-5} \pmod{p^2},
$$
$$
H_{p-1}(2,1,1)\equiv \frac{3}{5}pB_{p-5} \!\!\pmod{p^2}, \quad
H_{p-1}(1,1,1,2)\equiv H_{p-1}(1,4)\equiv B_{p-5} \!\pmod{p};
$$

\vspace{0.3cm}

(g) (\cite[Cor.~2.3, Prop.~7.3, 6.1]{zhao1}) for any prime $p>5,$
$$
H_{p-1}(-4)\equiv \frac{3}{4}pB_{p-5}, \quad
H_{p-1}(-3)\equiv -2H_{p-1}(1,-2)\equiv \frac{3}{2}\frac{H_{p-1}(1)}{p^2} \pmod{p^2},
$$
$$
H_{p-1}(2,-2)\equiv -2H_{p-1}(1,-3), \qquad 2H_{p-1}(1,1,-2)\equiv H_{p-1}(1,-3) \pmod{p}.
$$

\vspace{0.3cm}

Note that by (a), (b), we can replace the harmonic number $H_{p-1}(1)$ appearing in the congruences of Theorems \ref{t1}, \ref{t2},
\ref{t4}, \ref{t5}
by an appropriate expression in terms of Bernoulli numbers.

\section{Preliminaries} \label{s2}

\begin{lemma} \label{l1}
Let $p$ be a prime greater than $5.$ Then
\begin{equation}
p\sum_{k=1}^{p-1}\frac{H_{k-1}(2)}{k^2\binom{2k}{k}}\equiv\frac{1}{3}\frac{H_{p-1}(1)}{p^2} \pmod{p^2}.
\label{eq10}
\end{equation}
\end{lemma}
\begin{proof}
Note that this congruence modulo $p$ easily follows from Lemma 6.1 of \cite{mt}. Following the same scheme of the proof
and applying (\ref{eq02.5}), it is possible to derive the more exact congruence modulo $p^2.$ We reproduce the proof here
for completeness. In \cite[Th.~3.1]{t1} it was shown that
\begin{equation}
\sum_{k=0}^n\frac{\binom{n}{k}\binom{n+k-1}{k}}{\binom{2k}{k}}(-t)^k=\frac{(-1)^nv_n(t-2)}{2},
\label{eq11}
\end{equation}
where $\{v_n(x)\}_{n\ge 0}$ is the Lucas sequence defined in the Introduction. Taking $n=p,$  where
$p$ is an odd prime,
and noting that $v_n(x)=(-1)^nv_n(-x),$ we get
\begin{equation}
\sum_{k=0}^p\frac{\binom{p}{k}\binom{p+k-1}{k}}{\binom{2k}{k}}(-t)^k=\frac{v_p(2-t)}{2}.
\label{eq12}
\end{equation}
For $1\le k\le p-1,$ we have
\begin{equation}
(-1)^{k-1}\binom{p}{k}\binom{p+k-1}{k}=\frac{p^2}{k^2}\prod_{m=1}^{k-1}\left(1-\frac{p^2}{m^2}\right)
\equiv \frac{p^2}{k^2}\Bigl(1-p^2H_{k-1}(2)\Bigr) \pmod{p^6}.
\label{eq13}
\end{equation}
Now, since $\binom{2k}{k},$ for $p/2<k<p,$ is a multiple of $p$ but not of $p^2,$ from (\ref{eq12}), (\ref{eq13})
we obtain
$$
p^4\sum_{k=1}^{p-1}\frac{t^k H_{k-1}(2)}{k^2\binom{2k}{k}}\equiv
p^2\sum_{k=1}^{p-1}\frac{t^k}{k^2\binom{2k}{k}}+
\frac{v_p(2-t)+t^p-2}{2} \pmod{p^5}.
$$
Setting in the above congruence $t=1$ and taking into account (\ref{eq02.5}) we get for $p>5,$
$$
p\sum_{k=1}^{p-1}\frac{H_{k-1}(2)}{k^2\binom{2k}{k}}\equiv \frac{H_{p-1}(1)}{3p^2}+\frac{v_p(1)-1}{p^3}
\pmod{p^2}.
$$
Now noting that $v_p(1)=\alpha^p+\alpha^{-p},$ where $\alpha$ is a root of the polynomial
$x^2-x+1$ and $p$ is a prime greater than $3,$ i.e., $v_p(1)=2\cos(\pi p/3)$
and $p\equiv \pm 1 \pmod{6},$ we get $v_p(1)=1,$ and the lemma
follows.
\end{proof}
\begin{lemma} \label{l2}
Let $p$ be a prime greater than $5.$ Then
$$
p\sum_{k=1}^{p-1}\frac{(-1)^k H_{k-1}(2)}{k^3\binom{2k}{k}}\equiv \frac{L_p^2-1}{2p^4}+\frac{1}{p^3}
\Bigl({\mathcal L}_1(\varphi^2)+{\mathcal L}_1(\varphi^{-2})\Bigr)-\frac{12H_{p-1}(1)}{5p^3} \pmod{p^2},
$$
where $L_n$ is the $n$th Lucas number defined by the recurrence $L_0=2,$ $L_1=1,$ $L_n=L_{n-1}+L_{n-2},$ $n>1,$
$\varphi=\frac{1+\sqrt{5}}{2}$ is the golden ratio,  and
${\mathcal L}_1(x)=\sum_{k=1}^{p-1}x^k/k$ is the finite $1$-logarithm.
\end{lemma}
\begin{proof}
Rewrite identity (\ref{eq11}) in the form
$$
\sum_{k=1}^n\frac{\binom{n}{k}\binom{n+k-1}{k}}{\binom{2k}{k}}(-t)^{k}=\frac{v_n(2-t)-2}{2}.
$$
Dividing both sides by $t$ and integrating with respect to $t,$ we see from the relation (see
\cite[Lemma 5.1]{mt})
$$
\int_0^t\frac{v_n(2-t)-2}{t}\,dt=\frac{v_n(2-t)-2}{n}+2\sum_{k=1}^{n-1}\frac{v_k(2-t)-2}{k}
$$
that
\begin{equation}
\sum_{k=1}^n\frac{\binom{n}{k}\binom{n+k-1}{k}}{\binom{2k}{k}} \frac{(-t)^k}{k}=
\frac{v_n(2-t)-2}{2n}+\sum_{k=1}^{n-1}\frac{v_k(2-t)-2}{k}.
\label{eq14}
\end{equation}
Now taking $n=p$ in (\ref{eq14}), by (\ref{eq13}), we obtain
\begin{equation*}
\begin{split}
p^4\sum_{k=1}^{p-1}\frac{t^k H_{k-1}(2)}{k^3\binom{2k}{k}}&\equiv
p^2\sum_{k=1}^{p-1}\frac{t^k}{k^3\binom{2k}{k}}
+\frac{v_p(2-t)-2+t^p}{2p} \\
&+\sum_{k=1}^{p-1}\frac{v_k(2-t)}{k}
-2H_{p-1}(1) \pmod{p^5}.
\end{split}
\end{equation*}
Setting $t=-1$ in the above congruence and employing (\ref{eq02.5}) we deduce
$$
p\sum_{k=1}^{p-1}\frac{(-1)^k H_{k-1}(2)}{k^3\binom{2k}{k}}\equiv \frac{v_p(3)-3}{2p^4}+
\frac{1}{p^3}\sum_{k=1}^{p-1}\frac{v_k(3)}{k}-\frac{12H_{p-1}(1)}{5p^3} \pmod{p^2}.
$$
Since $v_k(3)=\varphi^{2k}+\varphi^{-2k}=L_{2k}=L_k^2+2,$ we conclude the proof.
\end{proof}
Observing that  ${\mathcal L}_1(\varphi^2)+{\mathcal L}_1(\varphi^{-2})=\sum_{k=1}^{p-1}L_{2k}/k$ and
$$
p^2\sum_{k=1}^{p-1}\frac{(-1)^k H_{k-1}(2)}{k^3\binom{2k}{k}}\equiv 0 \pmod{p}
$$\
 we get the following interesting corollary.
\begin{corollary} \label{c1}
Let $p$ be an odd prime and $L_n$ be the $n$th Lucas number. Then we have
$$
\sum_{k=1}^{p-1}\frac{L_{2k}}{k}\equiv \frac{1-L_p^2}{2p}+\frac{12}{5} H_{p-1}(1) \pmod{p^3}.
$$
(Note that $L_p\equiv 1\pmod{p}.$)
\end{corollary}

The following lemma refines the corresponding result from \cite[Th.~2.3]{t1}.
\begin{lemma} \label{l3}
Let $p>3$ be a prime. Then
\begin{equation}
H_{p-1}(1,2)\equiv \frac{6}{5}p^2B_{p-5}-H_{p-1}(2,1)\equiv -3\frac{H_{p-1}(1)}{p^2}+\frac{1}{2}p^2B_{p-5} \pmod{p^3}.
\label{eq14.1}
\end{equation}
\end{lemma}
\begin{proof}
The first congruence in (\ref{eq14.1}) easily follows from the identity
\begin{equation}
H_{k}(1)H_{k}(2)=H_{k}(1,2)+H_{k}(2,1)+H_{k}(3).
\label{h121}
\end{equation}
Indeed, by (a) we have that for any prime $p>3,$ $H_{p-1}(1)\equiv 0\pmod{p^2},$
$H_{p-1}(2)\equiv 0\pmod{p}$ and for $p>5,$
$H_{p-1}(3)\equiv -\frac{6}{5}p^2B_{p-5} \pmod{p^3},$ which implies
$$
H_{p-1}(2,1)+H_{p-1}(1,2)\equiv\frac{6}{5}p^2B_{p-5} \pmod{p^3}, \qquad p>5.
$$
To prove the second congruence in (\ref{eq14.1}), we consider the following identity (see \cite[proof of Th.~2.3]{t1}):
$$
\sum_{k=1}^n\frac{1}{k^2}=\sum_{1\le i\le j\le n}\frac{(-1)^{j-1}}{ij}\binom{n}{j}, \qquad n\in{\mathbb N}.
$$
Setting $n=p$ we get
\begin{equation*}
\begin{split}
H_{p-1}(2)&=p\sum_{1\le i\le j\le p-1}\frac{(-1)^{j-1}}{ij^2}\binom{p-1}{j-1}+\frac{H_{p-1}(1)}{p} \\
&\equiv p\sum_{1\le i\le j\le p-1}\frac{1-pH_{j-1}(1)+p^2H_{j-1}(1,1)}{ij^2}+\frac{H_{p-1}(1)}{p} \\
&\equiv pH_{p-1}(3)+pH_{p-1}(1,2)+\frac{H_{p-1}(1)}{p}-p^2H_{p-1}(1,3)-p^2\sum_{1\le i<j\le p-1}
\frac{H_{j-1}(1)}{ij^2} \\
&+p^3H_{p-1}(1,1,3)+p^3\sum_{1\le i<j\le p-1}\frac{H_{j-1}(1,1)}{ij^2} \pmod{p^4}.
\end{split}
\end{equation*}
Since
$$
\sum_{1\le i<j\le p-1}\frac{H_{j-1}(1)}{ij^2}=\sum_{j=1}^{p-1}\frac{H_{j-1}^2(1)}{j^2}
\quad\text{and}\,\,\,
\sum_{1\le i< j\le p-1}\frac{H_{j-1}(1,1)}{ij^2}=\sum_{j=1}^{p-1}\frac{H_{j-1}(1,1)H_{j-1}(1)}{j^2},
$$
then by the formulas
\begin{equation}
H_{n}^2(1)=2H_{n}(1,1)+H_{n}(2)
\label{111}
\end{equation}
and
$$
H_{n}(1,1)H_{n}(1)=3H_{n}(1,1,1)+H_{n}(2,1)+H_{n}(1,2),
$$
we get
\begin{equation}
\begin{split}
H_{p-1}(2)&\equiv pH_{p-1}(3)+pH_{p-1}(1,2)+\frac{H_{p-1}(1)}{p}-p^2H_{p-1}(1,3) \\
&-2p^2H_{p-1}(1,1,2)-p^2H_{p-1}(2,2)+p^3H_{p-1}(1,1,3)+3p^3H_{p-1}(1,1,1,2) \\[5pt]
&+p^3H_{p-1}(2,1,2)+p^3H_{p-1}(1,2,2)
\pmod{p^4}.
\end{split}
\label{h12}
\end{equation}
Now we can evaluate the right-hand side modulo $p^4$ using known congruences for multiple harmonic sums.
By (c)--(e), for any prime $p>5$ we have
\begin{equation}
H_{p-1}(2,2)\equiv -\frac{2}{5}pB_{p-5}, \qquad H_{p-1}(1,3)\equiv -\frac{9}{10}pB_{p-5} \pmod{p^2}
\label{5}
\end{equation}
and
$$
H_{p-1}(1,2,2)\equiv -\frac{3}{2}B_{p-5}, \quad
H_{p-1}(2,1,2)\equiv 0, \quad
H_{p-1}(1,1,3)\equiv \frac{1}{2}B_{p-5} \pmod{p}.
$$
Further by (b) and (f),
substituting  the above congruences in (\ref{h12}) and simplifying we conclude that for $p>5,$
$$
H_{p-1}(1,2)\equiv -3\frac{H_{p-1}(1)}{p^2}+\frac{1}{2}p^2B_{p-5} \pmod{p^3}.
$$
The validity of (\ref{eq14.1}) for $p=5$ can be easily checked by hand.
\end{proof}

\section{Proofs of the main results} \label{s3}

{\bf Proof of Theorem \ref{t1}.}

For any non-negative integers $n, k,$ consider the pair of functions
\begin{equation}
\begin{split}
F(n,k)&=\frac{(-1)^{n+k} (n-k-1)! \,k!^2}{(n+k+1)!}\, H_k(2), \qquad n\ge k+1, \\[3pt]
G(n,k)&=\frac{2(-1)^{n+k} (n-k)! k!^2}{(n+k+1)! (n+1)} \Bigl(H_k(2)-\frac{1}{(n+1)^2}\Bigr), \qquad n\ge k.
\end{split}
\label{eq15}
\end{equation}
By straightforward verification it is easy to check that $(F,G)$ is a WZ pair, i.e.,
\begin{equation}
F(n+1,k)-F(n,k)=G(n,k+1)-G(n,k)
\label{eq16}
\end{equation}
for any $n, k\ge 0,$ $n\ge k+1.$ Now putting $h(n):=\sum_{k=0}^{n-1}F(n,k),$ $n\ge 1,$ and summing
(\ref{eq16}) over $k=0,1,\dots,n-1$ we obtain
\begin{equation}
h(n+1)-h(n)=G(n,n)+F(n+1,n)-G(n,0).
\label{eq17}
\end{equation}
Again, summing (\ref{eq17}) over $n=1,2,\dots,N$ we get
$$
h(N+1)-h(1)=\sum_{n=1}^N(G(n,n)+F(n+1,n))-\sum_{n=1}^NG(n,0)
$$
which is equivalent to the following summation formula:
\begin{equation}
\sum_{n=0}^NG(n,0)=\sum_{n=0}^N(G(n,n)+F(n+1,n))-\sum_{k=0}^NF(N+1,k).
\label{eq18}
\end{equation}
Now substituting the WZ pair $(F,G)$ defined by (\ref{eq15}) in (\ref{eq18}), simplifying
and replacing $N$ by $N-1$ we get the identity
$$
4\sum_{k=1}^N\frac{1}{k^4\binom{2k}{k}}=3\sum_{k=1}^N\frac{H_{k-1}(2)}{k^2\binom{2k}{k}}
-2\sum_{k=1}^N\frac{(-1)^k}{k^4}+
\sum_{k=1}^N\frac{(-1)^{N+k}(N-k)! (k-1)!^2 H_{k-1}(2)}{(N+k)!}.
$$
Setting $N=p-1$ and observing that
\begin{equation}
\begin{split}
&(-1)^k\frac{(p-1-k)! (k-1)!^2}{(p-1+k)!}=\frac{1}{pk}\prod_{m=1}^k\Bigl(1-\frac{p}{m}\Bigr)^{-1}
\prod_{m=1}^{k-1}\Bigl(1+\frac{p}{m}\Bigr)^{-1} \\
&\equiv \frac{1}{pk}\Bigl(1+pH_k(1)+p^2(H_k^2(1)-H_k(1,1))\Bigr)\Bigr(1-pH_{k-1}(1)+p^2(H_{k-1}^2(1)-H_{k-1}(1,1))\Bigr) \\
&\equiv \frac{1}{pk}+\frac{1}{k^2}+\frac{pH_k(2)}{k} \pmod{p^2},
\end{split}
\label{23}
\end{equation}
where in the last congruence we used (\ref{111}),  we obtain
\begin{equation*}
\begin{split}
\qquad &\qquad \qquad\qquad\qquad \quad
4\sum_{k=1}^{p-1}\frac{1}{k^4\binom{2k}{k}}-3\sum_{k=1}^{p-1}\frac{H_{k-1}(2)}{k^2\binom{2k}{k}} \\
&\equiv
-2H_{p-1}(-4)+\frac{1}{p}H_{p-1}(2,1)+H_{p-1}(2,2)+p
\sum_{k=1}^{p-1}\frac{H_{k-1}(2)H_k(2)}{k} \pmod{p^2}.
\end{split}
\end{equation*}
Applying to the last sum the identity
$$
(H_k(2))^2=2H_k(2,2)+H_k(4)
$$
we get
\begin{equation}
\begin{split}
&\qquad\qquad\quad 4\sum_{k=1}^{p-1}\frac{1}{k^4\binom{2k}{k}}-3\sum_{k=1}^{p-1}\frac{H_{k-1}(2)}{k^2\binom{2k}{k}} \\
&\equiv
-2H_{p-1}(-4)+\frac{1}{p}H_{p-1}(2,1)+H_{p-1}(2,2) \\[3pt]
&+pH_{p-1}(2,3)+2pH_{p-1}(2,2,1)+pH_{p-1}(4,1) \pmod{p^2}.
\label{eq19}
\end{split}
\end{equation}
Now from (d) for $p\ge 5$ we have
\begin{equation}
H_{p-1}(2,3)\equiv -2B_{p-5}, \qquad H_{p-1}(4,1)\equiv -B_{p-5} \pmod{p}.
\label{eq20}
\end{equation}
Similarly, from (e) we obtain
$$
H_{p-1}(2,2,1)\equiv \frac{3}{2}B_{p-5} \pmod{p}, \qquad p>5.
$$
Finally, substituting the above congruences in (\ref{eq19}), by Lemma \ref{l3}, (g) and (\ref{5}), we
get the required congruence for all primes $p\ge 7.$ The validity of Theorem \ref{t1} for $p=5$
can be easily checked by straightforward verification. \qed

\vspace{2cm}

{\bf Proof of Theorem \ref{t2}.}

From Theorem \ref{t1} and Lemma \ref{l1} we easily get the first congruence of Theorem \ref{t2}. To prove
the second one, which is dual to the first congruence, it is sufficient to apply \cite[Th.~3.3]{t1}
or simply note that for $k=1,2,\dots, p-1$ we have
$$
\frac{p}{k\binom{2k}{k}}\equiv \frac{1}{2}\binom{2(p-k)}{p-k} \pmod{p}. \qed
$$

\vspace{0.3cm}

{\bf Proof of Theorem \ref{t3}.}

For non-negative integers $n,k$ define the pair of functions
\begin{equation*}
\begin{split}
F(n,k)&=\frac{(-1)^k (n-k-1)! k!^2}{2(k+1) (n+k+1)!} \,H_k(2), \qquad n\ge k+1, \\[5pt]
G(n,k)&=\frac{(-1)^k(n-k)! k!^2}{(n+k+1)!(n+1)^2}\left(H_k(2)-\frac{1}{(n+1)^2}\right), \qquad n\ge k.
\end{split}
\end{equation*}
It is easy to check that the pair $(F,G)$ is a WZ pair, i.e.,
$$
F(n+1,k)-F(n,k)=G(n,k+1)-G(n,k)
$$
for any $n,k\ge 0,$ $n\ge k+1.$ Applying formula (\ref{eq18}) to $(F,G)$ and replacing $N$
by $N-1$ we get the identity
$$
2\sum_{k=1}^N\frac{(-1)^{k-1}}{k^5\binom{2k}{k}}-\frac{5}{2}\sum_{k=1}^N\frac{(-1)^{k-1}H_{k-1}(2)}%
{k^3\binom{2k}{k}}=\sum_{k=1}^N\frac{1}{k^5}+\frac{1}{2}\sum_{k=1}^N
\frac{(-1)^k(N-k)! (k-1)!^2 H_{k-1}(2)}{k(N+k)!}.
$$
Setting $N=p-1$ and employing (\ref{23}) we get
\begin{equation*}
\begin{split}
2\sum_{k=1}^{p-1}\frac{(-1)^{k-1}}{k^5\binom{2k}{k}}&-\frac{5}{2}\sum_{k=1}^{p-1}\frac{(-1)^{k-1}H_{k-1}(2)}%
{k^3\binom{2k}{k}}\equiv H_{p-1}(5)+\frac{1}{2}\sum_{k=1}^{p-1}\frac{H_{k-1}(2)}{k}\left(
\frac{1}{pk}+\frac{1}{k^2}\right) \\
&\equiv H_{p-1}(5)+\frac{1}{2p}H_{p-1}(2,2)+\frac{1}{2}H_{p-1}(2,3) \pmod{p}.
\end{split}
\end{equation*}
Now taking into account (\ref{5}), (\ref{eq20}) and the congruence \cite[Cor.~5.1]{zhs} $H_{p-1}(5)\equiv 0 \pmod{p}$
for $p>5,$ we get the required statement. The validity of Theorem \ref{t3} for $p=5$ can be easily checked by hand. \qed

Theorem \ref{t4} follows immediately from Theorem \ref{t3} and Lemma \ref{l2}.

\section{A $p$-analogue of Zeilberger's series for $\zeta(2).$} \label{s4}

Let $k$ be a non-negative integer. Define the sequence $\{b_{m,k}\}_{m\ge 0}$ by the power series
expansion
$$
\prod_{j=1}^k\left(1+\frac{a}{j}\right)^{-2}=\sum_{m=0}^{\infty}b_{m,k}a^m \qquad\text{if}
\quad k\ge 1, \qquad |a|<1,
$$
and put $b_{0,0}=1,$ $b_{m,0}=0,$ $m\ge 1.$
\begin{lemma} \label{l4}
For any non-negative integers $m, k$ the sequence $\{b_{m,k}\}$ consists of rational numbers
satisfying the following recurrence:
\begin{equation}
b_{m,k}=\sum_{j=0}^m\frac{(-1)^j(j+1)}{k^j}\,b_{m-j,k-1}, \quad m\ge 0, \,\,\, k\ge 1,
\label{eq21}
\end{equation}
$$
b_{0,k}=1, \quad k\ge 0, \qquad b_{m,0}=0, \quad m\ge 1.
$$
Moreover, for any prime $p>3$ and a positive integer $m$ we have
\begin{equation*}
b_{m,p-1}\equiv\begin{cases}
0 \pmod{p^2},     & \quad \text{if} \quad m \quad \text{is odd}; \\
0 \pmod{p}, & \quad\text{if} \quad m \quad \text{is even}.
\end{cases}
\end{equation*}
The first few values of this sequence are as follows:
\begin{equation}
\begin{split}
b_{1,k}&=-2H_k(1), \qquad b_{2,k}=3H_k^2(1)-2H_k(1,1), \\[3pt]
b_{3,k}&=6H_k(1)H_k(1,1)-2H_k(1,1,1)-4H_k^3(1),
\label{eq22}
\end{split}
\end{equation}
$$
b_{4,k}=5H_k^4(1)+6H_k(1)H_k(1,1,1)+3H_k^2(1,1)-12H_k^2(1)H_k(1,1)-2H_k(1,1,1,1).
$$
\end{lemma}
\begin{proof}
Observing that
$$
\left(1+\frac{a}{k}\right)^{-2}=\sum_{j=0}^{\infty}\frac{(-1)^j(j+1)}{k^j}\,a^j, \qquad\quad |a|<1,
$$
and multiplying the series
$$
\left(1+\frac{a}{k}\right)^{-2}\cdot\sum_{m=0}^{\infty}b_{m,k-1} a^m=\sum_{m=0}^{\infty}b_{m,k} a^m
$$
we get the required recurrence (\ref{eq21}). Expanding for $k\ge 1,$
$$
\prod_{j=1}^k\left(1+\frac{a}{j}\right)^{-1}=\left(\sum_{j=0}^kH_k(\{1\}^j)a^j\right)^{-1}
=\sum_{m=0}^{\infty}c_{m,k}a^m, \qquad |a|<1,
$$
and using the usual multiplication of the series we get
$$
c_{0,k}=1, \qquad \sum_{j=0}^m H_k(\{1\}^j)c_{m-j,k}=0, \qquad m\ge 1,
$$
which yields the recurrence formula for the coefficients $c_{m,k}:$
\begin{equation}
c_{0,k}=1, \qquad c_{m,k}=-\sum_{j=1}^mH_k(\{1\}^j)c_{m-j,k}, \qquad m\ge 1.
\label{eq23}
\end{equation}
From (\ref{eq23}) it follows easily by induction on $m$ that
\begin{equation}
c_{m,p-1}\equiv\begin{cases}
0 \pmod{p^2}     & \quad \text{if} \quad m \quad \text{is odd}, \\
0 \pmod{p} & \quad\text{if} \quad m \quad \text{is even}.
\end{cases}
\label{eq24}
\end{equation}
Indeed, for $m=1$ we have from (\ref{eq23}) by Wolstenholme's theorem,
$c_{1,p-1}=-H_{p-1}(1)\equiv 0\pmod{p^2}.$ If $m>1$ is odd (even), then taking into account
that the numbers $j$ and $m-j$ have  different (the same) parity and
\begin{equation*}
H_{p-1}(\{1\}^j)\equiv\begin{cases}
0 \pmod{p^2}     & \quad \text{if} \quad j \quad \text{is odd}, \\
0 \pmod{p}  & \quad\text{if} \quad j \quad \text{is even},
\end{cases}
\end{equation*}
by recurrence (\ref{eq23}), we get the congruence (\ref{eq24}). Since
\begin{equation}
b_{m,k}=\sum_{j=0}^mc_{j,k}c_{m-j,k},
\label{eq25}
\end{equation}
by (\ref{eq24}), we get the required congruences for $b_{m,p-1}.$  Formulas (\ref{eq22}) can be readily
obtained from relations (\ref{eq23}) and (\ref{eq25}).
\end{proof}
\begin{proposition} \label{p1}
For any positive integers $m,n$ we have
$$
\sum_{k=1}^n\frac{(3k-2)b_{m,k-1}+2b_{m-1,k-1}}{k\binom{2k}{k}}=
-\frac{b_{m,n}}{\binom{2n}{n}}+\sum_{k=1}^n\frac{3kb_{m-2,k}+2b_{m-3,k}}{k^3\binom{2k}{k}},
$$
where $b_{-2,k}=b_{-1,k}:=0,$ and
\begin{equation}
\sum_{k=1}^n\frac{3k-2}{k\binom{2k}{k}}=1-\frac{1}{\binom{2n}{n}}.
\label{eq26}
\end{equation}
\end{proposition}
\begin{proof}
For a non-negative integer $m$ consider the difference
\begin{equation}
\frac{b_{m,k}}{\binom{2k}{k}}-\frac{b_{m,k-1}}{\binom{2k}{k}}=
\frac{b_{m,k}}{\binom{2k}{k}}-\frac{b_{m,k-1}}{\binom{2k-2}{k-1}}+\frac{(3k-2) b_{m,k-1}}{k\binom{2k}{k}}.
\label{eq27}
\end{equation}
Summing (\ref{eq27}) over $k=1,2,\dots,n$ we get
\begin{equation}
\sum_{k=1}^n\left(\frac{b_{m,k}}{\binom{2k}{k}}-\frac{b_{m,k-1}}{\binom{2k}{k}}\right)=
\frac{b_{m,n}}{\binom{2n}{n}}-b_{m,0}+\sum_{k=1}^n\frac{(3k-2) b_{m,k-1}}{k\binom{2k}{k}}.
\label{eq28}
\end{equation}
Putting $m=0$ in (\ref{eq28}) implies (\ref{eq26}). If $m\ge 1,$ then $b_{m,0}=0$ and by the recurrence
relation (\ref{eq21}), we have
\begin{equation}
\frac{b_{m,n}}{\binom{2n}{n}}+\sum_{k=1}^n\frac{(3k-2)b_{m,k-1}}{k\binom{2k}{k}}=
-\sum_{k=1}^n\frac{2b_{m-1,k-1}}{k\binom{2k}{k}}+\sum_{k=1}^n\sum_{j=2}^m\frac{(-1)^j(j+1)}{k^j\binom{2k}{k}}\,b_{m-j,k-1}.
\label{eq29}
\end{equation}
If $m=1,$ then the double sum in (\ref{eq29}) is empty and we get the desired identity. If $m\ge 2,$ then we note
that by (\ref{eq21}), we have
\begin{equation}
\begin{split}
3kb_{m-2,k}&+2b_{m-3,k}=3k\sum_{j=0}^{m-2}\frac{(-1)^j(j+1)}{k^j}\,b_{m-2-j,k-1}+
2\sum_{j=0}^{m-3}\frac{(-1)^j (j+1)}{k^j}\,b_{m-3-j,k-1} \\
&=3k b_{m-2,k-1}+
\sum_{j=0}^{m-3}\frac{(-1)^jb_{m-3-j,k-1}}{k^j}\,(2j+2-3j-6) \\
&=
3kb_{m-2,k-1}+\sum_{j=3}^m\frac{(-1)^j(j+1)b_{m-j,k-1}}{k^{j-3}}=
k^3\sum_{j=2}^m\frac{(-1)^j(j+1)b_{m-j,k-1}}{k^j}.
\end{split}
\label{eq30}
\end{equation}
Now from (\ref{eq29}) and (\ref{eq30}) we deduce the required statement.
\end{proof}
\begin{remark}
Note that formula (\ref{eq26}) also follows from \cite[Th.~5.4]{mt}.
\end{remark}
\begin{proposition} \label{p2}
Let $p>5$ be a prime. Then the following congruences are true:
\begin{eqnarray*}
p\sum_{k=1}^{p-1}\frac{3kb_{1,k}+2b_{0,k}}{k^3\binom{2k}{k}}&\equiv&\frac{4H_{p-1}(1)}{p} \pmod{p^3}, \\
p\sum_{k=1}^{p-1}\frac{3kb_{2,k}+2b_{1,k}}{k^3\binom{2k}{k}}&\equiv& \,\, 0  \quad\qquad\,\pmod{p^2}, \\
p\sum_{k=1}^{p-1}\frac{3kb_{3,k}+2b_{2,k}}{k^3\binom{2k}{k}}&\equiv& \,\, 0 \quad\qquad\,  \pmod{p}.
\end{eqnarray*}
\end{proposition}
\begin{proof}
For any non-negative integers $n, k$ define the pair of functions
\begin{equation*}
\begin{split}
F(n,k)&=\frac{(-1)^{n+k} k!^2 n!^2 (1+a)_{n-k-1}}{(n+k+1)! (1+a)_n^2}, \qquad\qquad\qquad n\ge k+1, \\[3pt]
G(n,k)&=\frac{(-1)^{n+k} k!^2 n!^2 (1+a)_{n-k}(2+2n+a)}{(n+k+1)! (1+a)^2_{n+1}}, \qquad n\ge k,
\end{split}
\end{equation*}
where $(a)_0=1,$ $(a)_m=a(a+1)\cdots(a+m-1),$ $m\ge 1,$ is the Pochhammer symbol.
It is easy to check that the pair $(F,G)$ satisfies (\ref{eq16}) for any $n,k\ge 0,$ $n\ge k+1.$
Applying summation formula (\ref{eq18}) and replacing $N$ by $N-1$ we get the identity:
\begin{equation}
\begin{split}
\sum_{k=1}^N\frac{(-1)^{k-1}\bigl(1+\frac{k}{k+a}\bigr)}{k^2\prod_{j=1}^k\bigl(1+\frac{a}{j}\bigr)}
&=\sum_{k=1}^N\frac{3k+2a}{k^3\binom{2k}{k}\prod_{j=1}^k\bigl(1+\frac{a}{j}\bigr)^2} \\
&+\prod_{j=1}^N\Bigl(1+\frac{a}{j}\Bigr)^{-2}\cdot\sum_{k=1}^N\frac{(-1)^{N+k}(1+a)_{N-k}(k-1)!^2}{(N+k)!}.
\end{split}
\label{eq31}
\end{equation}
Now if we expand the summands in powers of $a,$ it is easy to notice that
$$
\sum_{k=1}^N\frac{3k+2a}{k^3\binom{2k}{k}\prod_{j=1}^k\bigl(1+\frac{a}{j}\bigr)^2}
=\sum_{m=0}^{\infty}\left(\sum_{k=1}^N\frac{3kb_{m,k}+2b_{m-1,k}}{k^3\binom{2k}{k}}\right)a^m.
$$
Therefore, comparing coefficients of $a^m$ on both sides of (\ref{eq31}) leads to the
following family of identities for all $m\ge 1:$
\begin{equation}
\begin{split}
\sum_{k=1}^N\frac{3kb_{m,k}+2b_{m-1,k}}{k^3\binom{2k}{k}}&=\!\sum_{k=1}^N
\frac{(-1)^{k-1}\!}{k^3}\Bigl(2k\!\sum_{j=0}^mb_{m-j,k}H_{k-1}(\{1\}^j)+
\!\sum_{j=0}^{m-1}b_{m-1-j,k}H_{k-1}(\{1\}^j)\Bigr) \\
&-\sum_{k=1}^N\frac{(-1)^{N+k}(N-k)!(k-1)!^2}{(N+k)!}\sum_{j=0}^mH_{N-k}(\{1\}^j)b_{m-j,N},
\label{eq32}
\end{split}
\end{equation}
where $H_n(\{1\}^0):=1.$
In particular, for $m=1, 2$ we have
\begin{equation}
\begin{split}
&\sum_{k=1}^N\frac{3kb_{1,k}+2}{k^3\binom{2k}{k}}=\sum_{k=1}^N(-1)^k\Bigl(\frac{2H_{k-1}(1)}{k^2}+\frac{3}{k^3}
\Bigr) \\
&-\sum_{k=1}^N\frac{(-1)^{N+k}(N-k)!(k-1)!^2(b_{1,N}+H_{N-k}(1))}{(N+k)!},
\label{eq33}
\end{split}
\end{equation}
\begin{equation}
\begin{split}
&\sum_{k=1}^N\frac{3kb_{2,k}+2b_{1,k}}{k^3\binom{2k}{k}}=\sum_{k=1}^N\frac{(-1)^{k-1}}{k^2}\Bigl(2H^2_{k}(1)
-2H_k(1,1)+\frac{H_k(1)}{k}+\frac{1}{k^2}
\Bigr) \\
&-\sum_{k=1}^N\frac{(-1)^{N+k}(N-k)!(k-1)!^2}{(N+k)!}(H_{N-k}(1,1)+b_{1,N}H_{N-k}(1)+b_{2,N}).
\label{eq34}
\end{split}
\end{equation}
Setting $N=p-1$ in (\ref{eq33}) by (\ref{eq22}) and (\ref{23}), we obtain
\begin{equation*}
\begin{split}
p\sum_{k=1}^{p-1}\frac{3kb_{1,k}+2b_{0,k}}{k^3\binom{2k}{k}}&\equiv 2pH_{p-1}(1,-2)+3pH_{p-1}(-3) \\
&-
\sum_{k=1}^{p-1}\left(\frac{1}{k}+\frac{p}{k^2}+\frac{p^2H_k(2)}{k}\right)\Bigl(H_{p-1-k}(1)-2H_{p-1}(1)\Bigr) \pmod{p^3}.
\end{split}
\end{equation*}
Since
$$
H_{p-1-k}(1)-H_{p-1}(1)=\frac{1}{1-p}+\cdots+\frac{1}{k-p}\equiv H_k(1)+pH_k(2)+p^2H_k(3) \pmod{p^3},
$$
by (a), (g), for $p>5$ we have
\begin{equation*}
\begin{split}
p\sum_{k=1}^{p-1}\frac{3kb_{1,k}+2b_{0,k}}{k^3\binom{2k}{k}}&\equiv\frac{3H_{p-1}(1)}{p}-
\sum_{k=1}^{p-1}\frac{H_k(1)}{k}-p\sum_{k=1}^{p-1}\frac{H_k(2)}{k}-p\sum_{k=1}^{p-1}\frac{H_k(1)}{k^2} \\
&-p^2\sum_{k=1}^{p-1}\frac{H_k(3)}{k}-p^2\sum_{k=1}^{p-1}\frac{H_k(2)}{k^2}-p^2\sum_{k=1}^{p-1}
\frac{H_k(2)H_k(1)}{k} \pmod{p^3}.
\end{split}
\end{equation*}
Applying (\ref{h121}) for evaluation of the last sum, by (a),  we get
\begin{equation*}
\begin{split}
p\sum_{k=1}^{p-1}\frac{3kb_{1,k}+2b_{0,k}}{k^3\binom{2k}{k}}&\equiv\frac{3H_{p-1}(1)}{p}-H_{p-1}(1,1)
-H_{p-1}(2)-pH_{p-1}(2,1)-pH_{p-1}(1,2) \\
&-2p^2H_{p-1}(3,1)-2p^2H_{p-1}(2,2)-p^2H_{p-1}(1,3) \\
&-p^2H_{p-1}(1,2,1)-p^2H_{p-1}(2,1,1)  \pmod{p^3}.
\end{split}
\end{equation*}
Now taking into account that
$$
H_{p-1}(1,1)=\frac{1}{2}\Bigl(H_{p-1}^2(1)-H_{p-1}(2)\Bigr)\equiv -\frac{H_{p-1}(2)}{2}
\equiv \frac{H_{p-1}(1)}{p}\pmod{p^3}
$$
by (a)--(f) and Lemma \ref{l3}, we  easily get the first congruence of the Proposition.

Similarly, setting $N=p-1$ in (\ref{eq34}) by  (\ref{23}),  Lemma \ref{l4} and (a),  we obtain
\begin{equation*}
\begin{split}
p\sum_{k=1}^{p-1}\frac{3kb_{2,k}+2b_{1,k}}{k^3\binom{2k}{k}}&\equiv
p\sum_{k=1}^{p-1}\frac{(-1)^{k-1}}{k^2}\Bigl(2H_k^2(1)-2H_k(1,1)+\frac{H_k(1)}{k}+\frac{1}{k^2}\Bigr) \\
&-\sum_{k=1}^{p-1}\Bigl(\frac{1}{k}+\frac{p}{k^2}\Bigr)\Bigl(H_{p-1-k}(1,1)-2H_{p-1}(1,1)\Bigr) \pmod{p^2}.
\end{split}
\end{equation*}
Applying (\ref{111}) and the formula
$$
H_k(1,1)=H_{k-1}(1,1)+\frac{H_{k-1}(1)}{k}
$$
after simplifying we find that
\begin{equation*}
\begin{split}
p\sum_{k=1}^{p-1}&\frac{3kb_{2,k}+2b_{1,k}}{k^3\binom{2k}{k}}\equiv -2pH_{p-1}(1,1,-2)-3pH_{p-1}(1,-3)-4pH_{p-1}(-4) \\
&-2pH_{p-1}(2,-2)
-\sum_{k=1}^{p-1}\frac{H_{p-1-k}(1,1)}{k}-p\sum_{k-1}^{p-1}\frac{H_{p-1-k}(1,1)}{k^2}
\pmod{p^2}.
\end{split}
\end{equation*}
Since from (c), (f) it follows that
\begin{equation*}
\begin{split}
\sum_{k=1}^{p-1}\frac{H_{p-1-k}(1,1)}{k}=\sum_{k=1}^{p-1}\frac{H_{k-1}(1,1)}{p-k}&\equiv
-\sum_{k=1}^{p-1}H_{k-1}(1,1)\Bigl(\frac{1}{k}+\frac{p}{k^2}\Bigr) \\
&=
-H_{p-1}(1,1,1)-pH_{p-1}(1,1,2)\equiv 0 \pmod{p^2}
\end{split}
\end{equation*}
and
\begin{equation*}
\sum_{k=1}^{p-1}\frac{H_{p-1-k}(1,1)}{k^2}=\sum_{k=1}^{p-1}\frac{H_{k-1}(1,1)}{(p-k)^2}
\equiv\sum_{k=1}^{p-1}\frac{H_{k-1}(1,1)}{k^2}=H_{p-1}(1,1,2)\equiv 0\pmod{p},
\end{equation*}
we conclude, by (g), the second congruence of the Proposition.

Finally, from (\ref{eq32}) with $m=3,$  $N=p-1,$ by  (\ref{23}),  Lemma \ref{l4} and (a), (b), we have
\begin{equation*}
p\!\sum_{k=1}^{p-1}\frac{3kb_{3,k}+2b_{2,k}}{k^3\binom{2k}{k}}\equiv\!-\!\sum_{k=1}^{p-1}
\frac{H_{p-1-k}(\{1\}^3)}{k}\equiv\!\sum_{k=1}^{p-1}\frac{H_{k-1}(\{1\}^3)}{k}=H_{p-1}(\{1\}^4)\equiv 0
\!\!\!\!\pmod{p},
\end{equation*}
which completes the proof.
\end{proof}

\vspace{0.3cm}

{\bf Proof of Theorem \ref{t5}.}

For any non-negative integers $n,k$ define the pair of functions
$$
F(n,k)=\frac{k!^4n!^2(2n+3k+3)}{2(2k+1)!(k+n+1)!^2}, \qquad
G(n,k)=\frac{k!^4n!^2}{(2k)!(k+n+1)!^2}.
$$
It is straightforward to check that $(F,G)$ is a WZ pair.
Applying summation formula (\ref{eq18}) and replacing $N$ by $N-1$ we get
$$
\sum_{k=1}^N\frac{1}{k^2}=\sum_{k=1}^N\frac{21k-8}{k^3\binom{2k}{k}^3}-
\sum_{k=1}^N\frac{k!(k-1)!^3N!^2(2N+3k)}{(2k)!(k+N)!^2}.
$$
Setting $N=p-1$ and noting that for $1\le k\le p-1,$
\begin{equation*}
\begin{split}
\frac{k!(k-1)!^3(p-1)!^2(2p+3k-2)}{(2k)!(k+p-1)!^2}&=
\frac{2p+3k-2}{k\binom{2k}{k}}\frac{1}{p^2}\prod_{j=1}^{k-1}\Bigl(1+\frac{p}{j}\Bigr)^{-2} \\
&\equiv\frac{2p+3k-2}{k\binom{2k}{k}}\sum_{m=0}^5b_{m,k-1}p^{m-2} \pmod{p^3}
\end{split}
\end{equation*}
we get
\begin{equation*}
\begin{split}
p^3\sum_{k=1}^{p-1}\frac{21k-8}{k^3\binom{2k}{k}^3}&\equiv p^3H_{p-1}(2)+p\sum_{k=1}^{p-1}
\frac{3k-2}{k\binom{2k}{k}} \\
&+\sum_{m=1}^5p^{m+1}\sum_{k=1}^{p-1}
\frac{(3k-2)b_{m,k-1}+2b_{m-1,k-1}}{k\binom{2k}{k}} \pmod{p^6}.
\end{split}
\end{equation*}
Now by Propositions \ref{p1}, \ref{p2} and Lemma \ref{l4}, for $p>5$ we obtain
\begin{equation*}
\begin{split}
&p^3\!\sum_{k=1}^{p-1}\frac{21k-8}{k^3\binom{2k}{k}^3}\equiv p^3H_{p-1}(2)+p-\frac{p}{\binom{2p-2}{p-1}}
-\sum_{m=1}^5\!\left(\frac{b_{m,p-1}}{\binom{2p-2}{p-1}}-\!\sum_{k=1}^{p-1}\frac{3kb_{m-2,k}+2b_{m-3,k}}{k^3\binom{2k}{k}}
\right)\!p^{m+1} \\
&\equiv p^3H_{p-1}(2)+p-\frac{p}{\binom{2p-2}{p-1}}
-\sum_{m=1}^5\frac{b_{m,p-1}p^{m+1}}{\binom{2p-2}{p-1}}+3\sum_{k=1}^{p-1}\frac{p^3}{k^2\binom{2k}{k}}
+4p^2H_{p-1}(1) \pmod{p^6}.
\end{split}
\end{equation*}
Employing (\ref{eq02.5}), (b) and Lemma \ref{l4} we get
\begin{equation}
p^3\sum_{k=1}^{p-1}\frac{21k-8}{k^3\binom{2k}{k}^3}\equiv p+3p^2H_{p-1}(1)
-\frac{p}{\binom{2p-2}{p-1}}-\sum_{m=1}^4\frac{b_{m,p-1}p^{m+1}}{\binom{2p-2}{p-1}} \pmod{p^6}.
\label{eq35}
\end{equation}
Taking into account that
\begin{equation*}
\begin{split}
\frac{p}{\binom{2p-2}{p-1}}&=(2p-1)\prod_{j=1}^{p-1}\Bigl(1+\frac{p}{j}\Bigr)^{-1}\equiv
2p-1+pH_{p-1}(1)+p^2H_{p-1}(1,1)-2p^2H_{p-1}(1) \\
&-2p^3H_{p-1}(1,1)+p^3H_{p-1}(1,1,1)+p^4H_{p-1}(1,1,1,1) \pmod{p^6}
\end{split}
\end{equation*}
and
$$
H_{p-1}(1,1)=\frac{1}{2}(H_{p-1}^2(1)-H_{p-1}(2))\equiv -\frac{1}{2}H_{p-1}(2)
\equiv\frac{H_{p-1}(1)}{p}-\frac{1}{5}p^3B_{p-5} \pmod{p^4},
$$
by (\ref{eq22}) and (c), we get
\begin{eqnarray*}
-\frac{p}{\binom{2p-2}{p-1}}&\equiv& 1-2p-2pH_{p-1}(1)+4p^2H_{p-1}(1)
+\frac{4}{5}p^5B_{p-5} \pmod{p^6},  \\[3pt]
-\frac{pb_{1,p-1}}{\binom{2p-2}{p-1}}&\equiv& 4pH_{p-1}(1)-2H_{p-1}(1) \pmod{p^5},  \\[3pt]
-\frac{pb_{2,p-1}}{\binom{2p-2}{p-1}}&\equiv& 4pH_{p-1}(1,1)-2H_{p-1}(1,1)\equiv
4H_{p-1}(1)-\frac{2H_{p-1}(1)}{p}+\frac{2}{5}p^3B_{p-5} \pmod{p^4}, \\[3pt]
-\frac{pb_{3,p-1}}{\binom{2p-2}{p-1}}&\equiv& -2H_{p-1}(1,1,1)\equiv \frac{4}{5}p^2B_{p-5} \pmod{p^3}, \\[3pt]
-\frac{p^2b_{4,p-1}}{\binom{2p-2}{p-1}}&\equiv& -2pH_{p-1}(1,1,1,1)\equiv \frac{2}{5}p^2B_{p-5} \pmod{p^3}.
\end{eqnarray*}
Now substituting the above congruences in (\ref{eq35}) and simplifying we get the required statement. \qed

\vspace{0.2cm}

{\bf\small Acknowledgement.}
This work was done during our summer visit in 2011 to the Abdus Salam International Centre for Theoretical Physics (ICTP),
Trieste, Italy. The authors wish to thank the staff and, in particular,  the Head of the Mathematics Section of the ICTP,
Professor Ramadas Ramakrishnan, for their hospitality and the excellent working conditions.

\end{document}